\documentclass[12pt]{article}
\usepackage[centertags]{amsmath}
\usepackage{amsfonts}
\usepackage{amssymb}
\usepackage{amsthm}
\input{amssym.def}

\newtheorem{Definition}{Definition}[section]
\newtheorem{Theorem}[Definition]{Theorem}
\newtheorem{theorem}[Definition]{Theorem}

\newtheorem{Lemma}[Definition]{Lemma}

\newtheorem{Corollary}[Definition]{Corollary}

\numberwithin{equation}{section}

\newcommand{\lc}{\mathcal{L}}
\newcommand{\rc}{\mathcal{R}}

\newcommand{\dc}{\mathcal{D}}
\newcommand{\lp}{\mathcal{L}^+}
\newcommand{\rp}{\mathcal{R}^+}

\newcommand{\dpp}{\mathcal{D}^+}

\newcommand{\ld}{\mathcal{L}^{\bullet}}
\newcommand{\rd}{\mathcal{R}^{\bullet}}

\newcommand{\dd}{\mathcal{D}^{\bullet}}

\newcommand{\vld}{\mathbf{L}^{\bullet}}
\newcommand{\vrd}{\mathbf{R}^{\bullet}}
\newcommand{\vdd}{\mathbf{D}^{\bullet}}

\newcommand{\lzp}{\mathbf{LZ}^{+}}
\newcommand{\rzp}{\mathbf{RZ}^{+}}

\newcommand{\vd}{\mathbf{D}}
\newcommand{\vi}{\mathbf{I}}
\newcommand{\vn}{\mathbf{N}}
\newcommand{\vln}{\mathbf{LN}}
\newcommand{\vrn}{\mathbf{RN}}

\newcommand{\bi}{\mathbf{BI}}

\newcommand{\lqbi}{\mathbf{LQBI}}
\newcommand{\rqbi}{\mathbf{RQBI}}
\newcommand{\lnbd}{\mathbf{LNB}^{\bullet}}

\title{On the idempotent semirings such that $\dd$ is the least distributive lattice congruence}
\author{M. K. Sen$^{a}$, A. K. Bhuniya$^{b}$\footnote{Corresponding author} \ and R. Debnath$^c$}

\date{}

\begin{document}

\maketitle

\begin{center}
a) Department of Pure Mathematics, University of Calcutta, Kolkata, India \\
b) Department of Mathematics, Visva-Bharati, Santiniketan-731235, India. \\
c) Department of Mathematics, Kurseong college, Kurseong-734203, India.
\end{center}
{\it E-mail addresses}: senmk6@yahoo.com$^a$, anjankbhuniya@gmail.com$^b$ and rajib.d6@gmail.com$^c$.

\begin{abstract}
Here we describe the least distributive lattice congruence $\eta$ on an idempotent semiring in general and characterize the varieties $\vdd, \vld$ and $\vrd$ of all idempotent semirings such that $\eta=\dd, \ld$ and $\rd$, respectively. If $S \in \vdd [\vld, \vrd]$, then the multiplicative reduct $(S, \cdot)$ is a [left, right] normal band. Every semiring $S \in \vdd$ is a spined product of a semiring in $\vld$ and a semiring in $\vrd$ with respect to a distributive lattice.
\end{abstract}

\section{Introduction}
A \emph{semiring} $(S,+, \cdot)$ is an algebra with two binary operations + and $\cdot$ such that both the additive reduct $(S, +)$ and the multiplicative reduct $(S, \cdot)$ are semigroups and such that the following distributive laws hold:
\begin{align*}
 x(y+z) = xy+xz \; \textrm{and} \; (x+y)z = xz+yz.
\end{align*}
If moreover, both the reducts $(S, +)$ and $(S, \cdot)$ are bands, then $S$ is called an \emph{idempotent semiring}. Thus the class of all idempotent semirings is an equational class satisfying two additional identities:
\begin{align*}
x+x \approx x \; \textrm{and} \; x \cdot x \approx x.
\end{align*}
The variety of all idempotent semirings will be denoted by $\vi$.

We now fix the notation for some varieties of idempotent semirings and give their determining identity within $\vi$:
\begin{center}
\begin{table}[ht]
%\caption{Certain subvarieties of $\vi$}
\begin{tabular}{llll}
\hline
  Notation & Defining identity within $\vi$ \qquad\qquad & Notation & Defining identity within $\vi$ \\
  \hline
  $\mathbf{R}^{+}$ & $x+y+x \approx x$, & $\mathbf{R}^{\bullet}$ & $xyx \approx x$, \\
  $\mathbf{LZ}^{+}$ & $x+y \approx x$, & $\mathbf{RZ}^{+}$ & $x+y \approx y$, \\
  $\mathbf{LZ}^{\bullet}$ & $xy \approx x$, & $\mathbf{RZ}^{\bullet}$ & $xy \approx y$, \\
  $\mathbf{LNB}^{\bullet}$ & $xyz \approx xzy$, & $\mathbf{RNB}^{\bullet}$ & $xyz \approx yxz$, \\
  $\lqbi$ & $x+xy+x \approx x$, & $\rqbi$ & $x+yx+x \approx  x$, \\
  $\vln$ & $x+xyx \approx x$, & $\vrn$ & $xyx+x \approx x$,\\
  $\vn$ & $x+xyx+x \approx x$, & $\mathcal{S}l^+$ & $x+y \approx y+x$.
   \\ \hline
\end{tabular}
\end{table}
\end{center}

For an idempotent semiring $(S, +, \cdot)$ the Green's relations $\lc, \rc$ and $\dc$ on the additive [multiplicative] reduct $(S, +)[(S, \cdot)]$  will be denoted by $\lp, \rp$ and $\dpp [\ld, \rd, \dd]$. Since the multiplicative reduct $(S, \cdot)$ of an idempotent semiring $S$ is a band, $\dd, \ld$ and $\rd$ are given by: for $a, b \in S$,
\begin{align*}
& a \dd b \, \Leftrightarrow \, aba=a, \ bab=b, \\
& a \ld b \, \Leftrightarrow \, ba=b, \ ab=a \\
\textrm{and} \; & a \rd b \, \Leftrightarrow \, ab=b, \ ba=a.
\end{align*}
The relations $\dpp, \lp$ and $\rp$ are given dually.

The variety $\vi$ of all idempotent semirings contains the variety $\vd$ of all distributive lattices; and hence there is the least distributive lattice congruence on each idempotent semiring $S$. Throughout this article, we denote the least distributive lattice congruence on $S$ by $\eta_S$ or precisely by $\eta$. Then $(S/\eta, +)$ is a semilattice. Since $(S, +)$ is a band, $\dpp$ is the least congruence on the idempotent semiring $S$ such that $(S/\dpp, +)$ is a semilattice. Hence $\dpp \subseteq \eta$. In \cite{SGS}, Sen, Guo and Shum characterized  in depth the idempotent semirings $S$ such that $\eta = \dpp$.
\begin{Lemma} \cite{SGS}
Let $S$ be an idempotent semiring. Then the following conditions are equivalent.
\begin{enumerate}
\item
$\dpp$ is the least distributive lattice congruence on $S$;
\item
$S$ satisfies the identities:
\begin{align}
                x+xy+x & \approx x       \label{bi1} \\
\textrm{and} \; x+yx+x & \approx x;       \label{bi2}
\end{align}
\item
$S \in \mathbf{R}^{+} \circ \vd$.
\end{enumerate}
\end{Lemma}

An idempotent semiring $S$ is said to be a \emph{band semiring} if $\dpp$ is the least distributive lattice congruence on $S$, equivalently if it satisfies the identities (\ref{bi1}) and (\ref{bi2}). We will denote the variety of all band semirings by $\bi$. We define an idempotent semiring to be a \emph{left [right] quasi-band semiring} if it satisfies the identity (\ref{bi1}) [(\ref{bi2})]. Thus $\bi =\lqbi \cap \rqbi$. Additive reduct $(S, +)$ of a band semiring $S$ is a regular band \cite{WZG}.

The idempotent semirings for which $\eta = \dd$ have been characterized by Pastijn and Zhao \cite{PZ2000}. There are several articles characterizing the variety $\vi$ of idempotent semirings by Green's relations \cite{PG}, ~\cite{PZ2000},~\cite{SGS} - \cite{ZSG}.

We extend the problem to the idempotent semirings $S$ such that the least distributive lattice congruence $\eta$ on $S$ is either of the Green's relations $\lp$, $\rp$, $\ld$ and $\rd$. For this, we first describe explicitly the least distributive lattice congruence $\eta$ on an idempotent semiring in general, which makes it possible to characterize the idempotent semirings for which $\eta=\dpp$ $[\dd, \lp, \ld, \rp, \rd]$ in an unified and simple approach.

We present our results in two parts. The idempotent semirings $S$ such that $\eta=\dpp$ $[\lp, \rp]$ have been characterized in \cite{BD}. The following result has some use in the last section of this articel.
\begin{Lemma} \cite{BD}                              \label{lp is the least lat cong}
Let $S$ be an idempotent semiring. Then the following conditions are equivalent.
\begin{enumerate}
\item
$\lp$ is the least distributive lattice congruence on $S$;
\item
$S \in \vln$ and $\dd \subseteq \lp$;
\item
$S$ satisfies the identities:
\begin{equation}
x+yxy \approx x;                 \label{identity for lp}
\end{equation}
\item
$S \in \mathbf{LZ}^{+} \circ \vd$.
\end{enumerate}
\end{Lemma}

In Section 2 of this article, we characterize the least distributive lattice congruence on an idempotent semiring. The idempotent semirings $S$ such that $\eta=\dd$ $[\ld, \rd]$ have been studied in Section 3. In Section 4, we show that the multiplicative reduct $(S, \cdot)$ of an idempotent semiring $S$ such that $\eta = \dd$ is a normal band, and such a semiring $S$ is a spined product of an idempotent semiring $S_1$ such that $\eta = \ld$ and an idempotent semiring $S_2$ such that $\eta = \rd$ with respect to a distributive lattice.

We refer to \cite{howie} for the information concerning semigroup theory, \cite{Golan} for background on semirings and \cite{McKenzie} for notions concerning universal algebra and lattice theory.

\section{The least distributive lattice congruence on an idempotent semiring}
Let $S$ be an idempotent semiring. Define a binary relation $\sigma$ on $S$ by: for $a, b \in S$,
\begin{align*}
a \sigma b  & \ \ \textrm{if and only if} \\
            & \hspace{2cm} aba=aba+a+aba \ \textrm{and} \ bab=bab+b+bab.
\end{align*}
\begin{Lemma}               \label{congruence generated by sigma}
Let $S$ be an idempotent semiring. Then the congruence relation generated by $\sigma$ is the least distributive lattice congruence on $S$.
\end{Lemma}
\begin{proof}
Let $a, b \in S$. Then we have $(a+b)(b+a)(a+b)+a+b+(a+b)(b+a)(a+b) = (a+b)(b+a+a+b+b+a)(a+b) = (a+b)(b+a)(a+b)$, and similarly \ $(b+a)(a+b)(b+a)=(b+a)(a+b)(b+a)+a+b+(b+a)(a+b)(b+a)$, which implies that $(a+b) \sigma (b+a)$. Also $(ab) \sigma (ba)$ for all $a, b \in S$.

Now, \ $a(a+ab)a+a+a(a+ab)a=(a+aba)+a+(a+aba)=a+aba=a(a+ab)a$ \ and \ $(a+ab)a(a+ab)+a+ab+(a+ab)a(a+ab)=(a+ab)a(a+ab)$ implies that $(a+ab) \sigma a$ for all $a, b \in S$. Thus any congruence on $S$ which contains $\sigma$ is a distributive lattice congruence on $S$.

Consider a distributive lattice congruence $\rho$ on $S$ and $a, b \in S$ such that $a \sigma b$. Then
$aba=aba+a+aba$ implies that $b \rho (b+aba)=(b+aba+a+aba) \rho (b+a) \rho (a+b)$. Similarly from $bab=bab+b=bab$ we have $a \rho (a+b)$ and it follows that $a \rho b$. Thus $\sigma$ is contained in every distributive lattice congruence on $S$. Hence the result follows.
\end{proof}

Now we show that $\sigma$ is not transitive on a semiring in general. For this consider the following example \cite{GPZ} of an idempotent semiring,
\begin{center}
$\begin{array}{c|ccc}
  + & a & b & c \\
  \hline
  a & a & b & c \\
  b & b & b & b \\
  c & c & b & c
\end{array}$
\hspace{2cm}
$\begin{array}{c|ccc}
  \cdot & a & b & c \\
  \hline
  a & a & a & a \\
  b & b & b & b \\
  c & a & b & c
\end{array}$
\end{center}
Then $a \sigma b$ and $b \sigma c$ but $a \overline{\sigma} c$ shows that $\sigma$ is not transitive.

The transitive closure $\sigma^{\ast}$ of $\sigma$ is given by:
\begin{align*}
a \sigma^{\ast} b  & \ \ \textrm{if and only if there exists} \ x \in S \ \textrm{such that} \\
            & \hspace{2cm} axbxa=axbxa+a+axbxa \ \textrm{and} \ bxaxb=bxaxb+b+bxaxb.
\end{align*}

A semiring $(S, +, \cdot)$ is called \emph{almost idempotent} if $(S, +)$ is a semilattice and $a+a^2=a^2$ for every $a \in S$. In \cite{SB5}, we proved that $\sigma^{\ast}$ is the least distributive lattice congruence on an almost idempotent semiring. Since every idempotent semiring $S$ with commutative addition is an almost idempotent semiring, $\sigma^{\ast}$ is the least distributive lattice congruence on $S$. The proof that $\sigma^{\ast}$ remains the least distributive lattice congruence on an idempotent semiring $S$ even if the addition in $S$ is not commutative, is almost similar and so we omit the proof here.
\begin{Theorem}                                                       \label{eta}
Let $S$ be an idempotent semiring. Then $\sigma^{\ast}$ is the least distributive distributive congruence on $S$.
\end{Theorem}

Almost all varieties of idempotent semirings considered here are subvarieties of $\vn$. For this, the following result will have an important roll throughout the rest of this article.
\begin{Lemma}                   \label{equivalent N}
An idempotent semiring $S \in \vn$ if and only if it satisfies the identity:
\begin{equation}
xz+xyz+xz \approx xz.
\end{equation}
\end{Lemma}
\begin{proof}
First assume that $S \in \vn$ and $a, b, c \in S$. Then $ac = (a+abca+a)c=ac+abcac+ac=ac+abca(c+cabc+c)+ac=ac+abcac+abc+abcac+ac$ implies that $ac+abc+abcac+ac=ac$ [add $abc+abcac+ac$ to both sides from right]. Similarly this implies that $ac = ac+abc+ac$.

Converse follows directly.
\end{proof}
\begin{Theorem}                        \label{sigma is the ldlc}
If an idempotent semiring $S \in \vn$, then $\sigma$ is the least distributive lattice congruence on $S$.
\end{Theorem}
\begin{proof}
Let $a, b \in S$ be such that $a \eta b$. Then there exists $x \in S$ such that $axbxa=axbxa+a+axbxa$ and $bxaxb=bxaxb+b+bxaxb$. Since $S \in \vn$, $ab+axb+ab=ab$ and $ba+bxa+ba=ba$, by Lemma \ref{equivalent N}. Then we have
\begin{align*}
  aba & = a(ba+bxa+ba) &&  \\
      & = aba+bxa+aba && \\
      & = aba+(ab+axb+ab)xa+aba && \\
      & = aba+abxa+axbxa+abxa+aba && \\
      & = aba+abxa+axbxa+a+axbxa+abxa+aba && \\
      & = aba+abxa+axbxa+a+(aba+abxa+ && \\
      & \hspace{3cm} axbxa+a+axbxa+abxa+aba) && \text{since $(S, +)$ is a band} \\
      & = aba+abxa+axbxa+a+aba && \\
      & = aba+a+aba && \text{since $(S, +)$ is a band}.
\end{align*}
Similarly, $bab=bab+b+bab$. Thus $a \sigma b$ and so $\eta \subseteq \sigma$. Also $\sigma$ is contained in every distributive lattice congruence. Hence $\sigma = \eta$ is the least distributive lattice congruence on $S$.
\end{proof}

\section{The idempotent semirings such that $\eta=\dd, \ld, \rd$}
In this section we show that the class of all idempotent semirings for which $\eta = \dd [\ld, \rd]$ is an equational class and so a variety. Here we find several systems of identities defining these varieties.

On an idempotent semiring $(S, +, \cdot)$ one may introduce the relations $\leq^{l}_{+}, \, \leq^{l}_{\cdot}, \leq^{r}_{+}, \, \leq^{r}_{\cdot}$ and $\leq_{+}, \, \leq_{\cdot}$ by the following: for $a, b \in S$,
\begin{align*}
& a \leq^{l}_{+} b \hspace{.5cm}  \Leftrightarrow \hspace{.5cm} b=a+b \hspace{1cm}
a \leq^{l}_{\cdot} b \hspace{.5cm} \Leftrightarrow \hspace{.5cm} a=ba \\
& a \leq^{r}_{+} b \hspace{.5cm}  \Leftrightarrow \hspace{.5cm} b=b+a \hspace{1cm}
a \leq^{r}_{\cdot} b \hspace{.5cm} \Leftrightarrow \hspace{.5cm} a=ab \\
& \hspace{1.4cm} \leq_{+} = \leq^{l}_{+} \cap \leq^{r}_{+} \hspace{.5cm} \textrm{and} \hspace{.5cm}
\leq_{\cdot} = \leq^{l}_{\cdot} \cap \leq^{r}_{\cdot}
\end{align*}
The relations $\leq^{l}_{+}, \, \leq^{l}_{\cdot}, \leq^{r}_{+}, \, \leq^{r}_{\cdot}$ are quasi-orders and the relations $\leq_{+}$ and $\leq_{\cdot}$ are partial orders~\cite{nambooripad}.

In \cite{PZ2000}, Pastijn and Zhao characterized the idempotent semirings $S$ such that $\dd$ is the least distributive lattice congruence on $S$. Following result is already proved in \cite{PZ2000}. Use of Theorem \ref{eta} shorten the proof which we would like to include here.
\begin{Theorem}                              \label{ddot least lat cong}
Let $S$ be an idempotent semiring. Then the following conditions are equivalent.
\begin{enumerate}
\item
$\dd$ is the least distributive lattice congruence on $S$;
\item
$S \in \vn$ and $\dpp \subseteq \dd$;
\item
$S$ satisfies the identity
\begin{equation}                          \label{identity for ddot}
x \approx xyx+x+xyx.
\end{equation}
\end{enumerate}
\end{Theorem}
\begin{proof}
$\mathbf{(1) \ \Rightarrow \ (2):}$ \ Let $a, b \in S$. Since $\dd$ is a distributive lattice congruence, $(a+ab) \dd a$. Then we have
\begin{align*}
 a=a(a+ab)a \, \Rightarrow \, & a=a+aba \\
               \Rightarrow \, & a=a+aba+a
\end{align*}
and so $S \in \vn$.

Now let $a, b \in S$ such that $a \dpp b$. Then $a=a+b+a$ and $b=b+a+b$ implies that $bab=bab+b+bab$ and $aba=aba+a+aba$, and so $a \sigma b$. Thus $\dpp \subseteq \sigma \subseteq \eta=\dd$.
\\ $\mathbf{(2) \ \Rightarrow \ (3):}$ \ Similar to the proof of the Theorem 2.17 \cite{PZ2000}.
\\ $\mathbf{(3) \ \Rightarrow \ (1):}$ \ Let $a, b \in S$ such that $a \eta b$. Since $S$ satisfied the identity (\ref{identity for ddot}, it follows that $S \in \vn$ and so $\eta=\sigma$, by Theorem \ref{sigma is the ldlc}. Hence we have
$$ aba=aba+a+aba \, \textrm{and} \, bab=bab+b+bab. $$
Since $S$ satisfies the identity (\ref{identity for ddot}), it follows that $a=aba$ and $b=bab$. Thus $a \dd b$ and so $\eta \subseteq \dd$. Also $\dd \subseteq \sigma = \eta$. Therefore $\dd=\eta$ and so $\dd$ is the least distributive lattice congruence on $S$.
\end{proof}

Now wecharacterize the idempotent semirings for which $\ld$ is the least distributive lattice congruence. First we prove the following lemma.
\begin{Lemma}                                               \label{bi1 iff n and}
Let $S$ be an idempotent semiring. Then the following conditions are equivalent.
\begin{enumerate}
\item
$S$ satisfies the identity (\ref{bi1});
\item
$S \in \vn$ and $\rd \subseteq \dpp$.
\end{enumerate}
\end{Lemma}
\begin{proof}
$\mathbf{(1) \, \Rightarrow \, (2):}$ \, Let $a, b \in S$ such that $a \rd b$. Then $a=ba$ and $b=ab$. Now $a=a+ab+a=a+b+a$ and $b=b+ba+b=b+a+b$ implies that $a \dpp b$. Thus $\rd \subseteq \dpp$. Also it follows trivially that $S \in \vn$.
\\$\mathbf{(2) \, \Rightarrow \, (1):}$ \, Let $a, b \in S$. Then $ab \rd aba$ implies that $ab \dpp aba$. This implies that
\begin{align*}
& aba+ab+aba=aba \\
\Rightarrow \, & a+aba+ab+aba+a=a+aba+a \\
\Rightarrow \, & a+aba+ab+aba+a=a, \hspace{1cm} (\textrm{since} \, S \in \vn) \\
\Rightarrow \, & a+aba+ab+aba+a=a+aba+ab+a \\
\Rightarrow \, & a=a+aba+ab+a \\
\Rightarrow \, & a+ab+a=a+aba+ab+a \\
\Rightarrow \, & a+ab+a=a
\end{align*}
Thus $S$ satisfies the identity (\ref{bi1}).
\end{proof}
\begin{Theorem}                  \label{ldot least lat cong}
Let $S$ be an idempotent semiring. Then the following conditions are equivalent.
\begin{enumerate}
\item
$\ld$ is the least distributive lattice congruence on $S$;
\item
$\dpp \subseteq \ld$ and $S$ satisfies the identity (\ref{bi1});
\item
$S \in \vn$ and $\rd \subseteq \dpp \subseteq \ld$;
\item
$\leq^{l}_{\cdot} \subseteq \leq_{+}$ for $S$;
\item
$S$ satisfies the identity
\begin{equation}
 x \approx xy+x+xy;                                      \label{identity for ldot}
\end{equation}
\item
$S$ satisfies the identity
\begin{equation}
 x \approx x(y+x+y).
\end{equation}
\end{enumerate}
\end{Theorem}
\begin{proof}
Equivalence of (2) and (3) follows from the Lemma \ref{bi1 iff n and}. Equivalence of (5) and (6) follows trivially.
\\ $\mathbf{(1) \, \Rightarrow \, (2):}$ \, Let $a, b \in S$. Then $(a+ab) \ld a$ implies that $a=a(a+ab)=a+ab$ and so $a=a+ab+a$ for all $a, b \in S$. Thus $S$ satisfies the identity (\ref{bi1}). Since $\dpp$ is the least semilattice congruence on the additive reduct $(S, +)$, $\dpp \subseteq \ld$.
\\ $\mathbf{(2) \, \Rightarrow \, (5):}$ \, Let $a, b \in S$. Then $(a+b) \dpp (b+a)$ and so $(a+b) \ld (b+a)$. Therefore
\begin{align*}
& a+b=(a+b)(b+a)=ab+a+b+ba \\
\textrm{and} \; \; & b+a=(b+a)(a+b)=ba+b+a+ab
\end{align*}
Also $S$ satisfies the identity (\ref{bi1}). Thus we have
\begin{align*}
& a=a(a+b)+a(b+a) \\
\Rightarrow \, & a=a(ab+a+b+ba)+a(ba+b+a+ab) \\
\Rightarrow \, & a=ab+a+ab+aba+ab+a+ab \\
\Rightarrow \, & ab+a+ab=ab+a+ab+aba+ab+a+ab \\
\Rightarrow \, & a=ab+a+ab
\end{align*}
for all $a, b \in S$.
\\ $\mathbf{(5) \, \Rightarrow \, (1):}$ \, For every $a, b \in S$, we have $a=ab+a+ab$. This implies that $a=aba+a+aba$. Therefore $\dd$ is the least distributive lattice congruence on $S$, by Theorem~\ref{ddot least lat cong}. Also for every $a, b \in S$, \,
\begin{align*}
& a=ab+a+ab  \\
\Rightarrow \, & a=a+ab+a+ab+a  \\
\Rightarrow \, & a=a+ab+a
\end{align*}
Now let $a, b \in S$ such that $a \dd b$. Then $a=aba$ and $b=bab$. Now
\begin{align*}
& a=aba \\
\Rightarrow \, & a+ab+a=aba+ab+aba \\
\Rightarrow \, & a=ab
\end{align*}
Similarly $b=ba$. Hence $a \ld b$ and so $\dd=\ld$. Thus $\ld$ is the least distributive lattice congruence on $S$.
\\ $\mathbf{(4) \, \Rightarrow \, (5):}$ \, Let $a, b \in S$. Then $ab \leq^{l}_{\cdot} a$. This implies that $ab \leq_{+} a$. Then $ab+a=a=a+ab$ and so $a=ab+a+ab$. Thus $S$ satisfies the identity (\ref{identity for ldot}.
\\ $\mathbf{(5) \, \Rightarrow \, (4):}$ \, Let $a, b \in S$ such that $a \leq^{l}_{\cdot} b$. Then $a=ba$. Also $a, b \in S$ implies that $b=ba+b+ba$. This implies that
\begin{align*}
& b=a+b+a \\
\Rightarrow \, & b+a=b=a+b
\end{align*}
and so $a \leq_{+} b$. Therefore $\leq^{l}_{\cdot} \subseteq \leq_{+} $.
\end{proof}

The left-right dual of this result is stated as follows. Since the proof is similar to the above theorem, we omit.
\begin{Theorem}                                                     \label{rdot least lat cong}
Let $S$ be an idempotent semiring. Then the following conditions are equivalent.
\begin{enumerate}
\item
$\rd$ is the least distributive lattice congruence on $S$;
\item
$\dpp \subseteq \rd$ and $S$ satisfies the identity (\ref{bi2});
\item
$S \in \vn$ and $\ld \subseteq \dpp \subseteq \rd$;
\item
$\leq^{r}_{\cdot} \subseteq \leq_{+}$ for $S$;
\item
$S$ satisfies the identity
\begin{equation}
x \approx yx+x+yx;                                   \label{identity for rdot}
\end{equation}
\item
$S$ satisfies the identity
\begin{equation}
 x \approx (y+x+y)x.
\end{equation}
\end{enumerate}
\end{Theorem}

\section{Joins and Malcev's Products:}
The variety of all idempotent semirings $S$ such that $\dd [\ld, \rd]$ is the least distributive lattice congruence on $S$ will be denoted by $\mathbf{D}^{\bullet} [\mathbf{L}^{\bullet}, \mathbf{R}^{\bullet}]$. Thus it follows, by Theorem \ref{ddot least lat cong}, Theorem \ref{ldot least lat cong} and Theorem \ref{rdot least lat cong}, that the varieties $\mathbf{D}^{\bullet}, \mathbf{L}^{\bullet}$ and $\mathbf{R}^{\bullet}$ are determined by the additional identities (\ref{identity for ddot}), (\ref{identity for ldot}) and (\ref{identity for rdot}), respectively.

For subvarieties $\mathbf{V}$ and $\mathbf{W}$ of $\vi$, the {\it Mal'cev product} $\mathbf{V} \circ \mathbf{W}$ of $\mathbf{V}$ and $\mathbf{W}$ (within $\vi$) is the class of all idempotent semirings $S$ on which there exists a congruence $\rho$ such that $S/\rho \in \mathbf{W}$ and such that the $\rho$-classes belong to $\mathbf{V}$.

In this section, different characterizations of the varieties $\mathbf{D}^{\bullet}, \mathbf{L}^{\bullet}$ and $\mathbf{R}^{\bullet}$ by Malcev's product are given.
\begin{Theorem}                             \label{vld=lzd v vd}
\begin{enumerate}
\item[(i)]
$\mathbf{L}^{\bullet} = \mathbf{LZ}^{\bullet} \circ \vd$.
\item[(ii)]
$\mathbf{R}^{\bullet} = \mathbf{RZ}^{\bullet} \circ \vd$.
\end{enumerate}
\end{Theorem}
\begin{proof}
(i) \ First we assume that $S \in \mathbf{LZ}^{\bullet} \circ \mathbf{D}$. Then there exists a congruence relation $\delta$ on $S$ such that $S/\delta \in \vd$ and such that each $\delta$-class belongs to $\mathbf{LZ}^{\bullet}$. Let $a, b \in S$. Since $S/\delta \in \vd$, we have $(a+ab) \delta a$ and $(ab+a) \delta a$. Since the $\delta$-class of $a$ is a left zero band for the multiplication, we have
$$ a=a(ab+a)=ab+a \hspace{.5cm} \textrm{and} \hspace{.5cm} a=a(a+ab)=a+ab, $$
and so $a=(ab+a)+(a+ab)=ab+a+ab$. Thus $S$ satisfies the identity \ref{identity for ldot}. Hence $S \in \mathbf{L}^{\bullet}$, by Theorem~\ref{ldot least lat cong}. Thus $\mathbf{LZ}^{\bullet} \circ \vd \subseteq \mathbf{L}^{\bullet}$. The reverse inclusion follows trivially. Hence $\mathbf{L}^{\bullet} = \mathbf{LZ}^{\bullet} \circ \vd$.
\\ (ii) \ The proof is left-right dual to (i).
\end{proof}

Thus for $S \in \mathbf{L}^{\bullet}$ the multiplicative reduct $(S, \cdot)$ is a left regular band and for $S \in \mathbf{R}^{\bullet}$ the multiplicative reduct $(S, \cdot)$ is a right regular band.
\begin{Lemma}
For an idempotent semiring $S$ the following conditions are equivalent:
\begin{enumerate}
\item \vspace{-.3cm}
$S \in \vln$.
\item \vspace{-.4cm}
$\dd$ is a congruence on $S$ and $S/\dd \in \lzp \circ \vd$.
\end{enumerate}
\end{Lemma}
\begin{proof}
$\mathbf{(1) \, \Rightarrow \, (2):}$ \, Let $S \in \vln$. Then $S \in \vn$ and so $\dd$ is a congruence on $S$, by Theorem 2.11~\cite{PZ2000}. Hence $S/\dd \in \vln$ and $\dd_{S/\dd} \subseteq \lp_{S/\dd}$. Thus $S/\dd \in \lzp \circ \vd$.
\\ $\mathbf{(2) \, \Rightarrow \, (1):}$ \, Assume that $S/\dd \in \lzp \circ \vd$. Then $\lp_{S/\dd}$ is the least distributive lattice congruence on $S/\dd$, by Lemma \ref{lp is the least lat cong}. Let $\overline{a}$ be the $\dd$-class of $a$ in $S$. Then for all $a, b \in S$, we have
\begin{align*}
(\overline{a}+\overline{a}\overline{b}) \lp \overline{a} \, & \Rightarrow \, \overline{a}=\overline{a+ab} \\
                                        & \Rightarrow \, a=a+aba
\end{align*}
and so $S \in \vln$.
\end{proof}
\begin{Theorem}
\begin{enumerate}
\item[(i)]
$\vln = \mathbf{R}^{\bullet} \circ (\lzp \circ \vd)$.
\item[(ii)]
$\vrn = \mathbf{R}^{\bullet} \circ (\rzp \circ \vd)$.
\end{enumerate}
\end{Theorem}
\begin{proof}
(Revise) \; \; Let $S \in \vln$. Then by the above lemma, $S/\dd \in \lzp \circ \vd$. Hence $\vln \subseteq \mathbf{R}^{\bullet} \circ (\lzp \circ \vd)$. If $S \in \mathbf{R}^{\bullet} \circ (\lzp \circ \vd)$ then there exists a congruence $\rho$ on $S$ such that $S/\rho \in \lzp \circ \vd$ and each $\rho$-class is in $\mathbf{R}^{\bullet}$. Hence  $S/\rho \in \vln$, by Lemma \ref{lp is the least lat cong}, and so $(a+aba) \rho a$ for all $a, b \in S$.
Since each $\rho$-class is in $\mathbf{R}^{\bullet}$, it follows that $a=a+aba$. Thus $\vln = \mathbf{R}^{\bullet} \circ (\lzp \circ \vd)$.
\end{proof}

Let $S_1$ and $S_2$ be two semirings and $D$ a distributive lattice. If there are two homomorphisms $\phi_1 : S_1 \longrightarrow D$ and $\phi_2 : S_2 \longrightarrow D$ onto $D$ then the semiring $S = \{(s_1, s_2) \in S_1 \times S_2 \mid \phi_1(s_1)=\phi_2(s_2)\}$ is called a spined product of the two semirings $S_1$ and $S_2$ with respect to the spine $D$.

We show that every semiring $S \in \mathbf{D}^{\bullet}$ is a spined product of a semiring $S_1 \in \mathbf{L}^{\bullet}$ and a semiring $S_2 \in \mathbf{R}^{\bullet}$. For this first we prove that the multiplicative reduct of each idempotent semiring  $S \in \mathbf{D}^{\bullet}$ is a normal band. We do the groundwork by providing a sequence of useful lemmas.
\begin{Lemma}
Every idempotent semiring satisfies the following two identities:
\begin{align}
& xyzx \approx xyzx+xyxzx+xyzx                                      \label{regband1}
\\ \textrm{and} \hspace{.5cm} & xyxzx \approx xyxzx+xyzx+xyxzx.             \label{regband2}
\end{align}
\end{Lemma}
\begin{proof}
Let $S$ be an idempotent semiring. Then $\dpp$ is a congruence on $S$ and $S/ \dpp \in \mathcal{S}l^+$. Hence the multiplicative reduct $(S/ \dpp, \cdot)$ is a regular band~\cite{PZ2005}. Then for all $a, b, c \in S$, \ $abca \dpp abaca$ and so
\begin{align*}
& abca=abca+abaca+abca
\\ \textrm{and} \hspace{.5cm} & abaca=abaca+abca+abaca
\end{align*}
\end{proof}
\begin{Lemma}
For an idempotent semiring $S$ the following conditions are equivalent:
\begin{enumerate}
\item 
$S \in \mathbf{D}^{\bullet}$;
\item 
$S$ satisfies both the identities:
\begin{align*}
xz \approx xz+xyz \hspace{.5cm} \textrm{and} \hspace{.5cm} xz \approx xyz+xz;
\end{align*}
\item 
$S$ satisfies the identity:
\begin{equation}                                  \label{xz=xyz+xz+xyz}
xz \approx xyz+xz+xyz.
\end{equation}
\end{enumerate}
\end{Lemma}
\begin{proof}
It is clear that ${\textbf (2) \Rightarrow (3)}$ and ${\textbf (3) \Rightarrow (1)}$. Assume that (1) holds. Then $xz\approx (xyzx+x+xyzx)(zxyz+z+zxyz)\approx xyzxzxyz+xyzxz+xyzxzxyz+xzxyz+xz+xzxyz+xyzxzxyz+xyzxz+xyzxzxyz\approx xyz+xyzxz+xyzxzxyz+xzxyz+xz+xzxyz+xyzxzxyz+xyzxz+xyz$. This implies that $xz \approx xz+xyz$ and $xz \approx xyz+xz$.
\end{proof}

\begin{Lemma}
If $S\in\vdd$ then $S$ satisfies the identity:
\begin{equation}
xyzx\approx xzyx+xyzx+xzyx.                             \label{nbd primary}
\end{equation}
\end {Lemma}
\begin{proof}
Let $S \in \vdd$. Then for $x,y,z \in S$,
\begin{align*}
xyzx & \approx (xzy+xy+xzy)zx \\
     & \approx xzyzx+xyzx+xzyzx \\
     & \approx xzy(zyx+zx+zyx)+xyzx+xzy(zyx+zx+zyx) \\
     & \approx xzyx+xzyzx+xzyx+xyzx+xzyx+xzyzx+xzyx,
\end{align*}
which implies that $xyzx+xzyx \approx xyzx$ and hence $xzyx+xyzx+xzyx \approx xyzx$.
\end{proof}

\begin{theorem}
For $S \in \vdd$ the multiplicative reduct $(S,\cdot)$ is a normal band.
\end{theorem}
\begin{proof}
Let $S \in \vdd$ and $x,y,z,\in S$. Then $xyzx \approx xzyx+xyzx+xzyx$. Also we have
\begin{align*}
& xzyx \approx xzyx+xzyx+xzyx \\
\Rightarrow \;& xzyx \approx xzyx+xyzx+xzyx+xyzx+xzyx  \\
\Rightarrow \;& xzyx+xyzx+xzyx \approx xzyx \\
\Rightarrow \;& xyzx \approx xzyx.
\end{align*}
\end{proof}

\begin{theorem}
For every idempotent $S$ the following conditions are equivalent:
\begin{itemize}
\item[1.]
$S \in \lnbd \cap \vdd$;
\item[2.]
$S$ satisfies the identity:  $$xz\approx xzy+xz+xzy;$$
\item[3.]
$S \in \mathbf{L}^{\bullet}$.
\end{itemize}
\end{theorem}
\begin{proof}
Equivalence of $(2)$ and $(3)$ is trivial.
\\ \textbf{$(1) \Rightarrow (2)$ :} \; Assume that $S \in \lnbd \cap \vdd$. Then we have
$$ xz \approx xyz+xz+xyz. $$
Since $S \in \lnbd$,
$$ xz \approx xzy+xz+xzy. $$
\textbf{$(2) \Rightarrow (1)$ :} \; Let $x,y,z \in S$. Then $xz \approx xzy+xz+xzy$ and $xz \approx xyz+xz+xyz$. Also we have
\begin{align*}
& xyz \approx xyzy+xyz+xyzy \\
\Rightarrow \; & xyz \approx (xzy+xy+xzy)zy+xyz+(xzy+xy+xzy)zy \\
\Rightarrow \; & xyz \approx xzy+xyzy+xzy+xyz+xzy+xyzy+xzy  \\
\Rightarrow \; & xzy+xyz+xzy \approx xyz.
\end{align*}
Now $xzy \approx xzy+xzy+xzy$ implies that
\begin{align*}
& xzy \approx xzy+xyzy+xzy+xyzy+xzy \\
\Rightarrow \; & xzy \approx xzy+xyzy+xzy \\
\Rightarrow \; & xzy \approx xzy+xyzyz+xyzy+xyzyz+xzy \\
\Rightarrow \; & xzy \approx xzy+xyz+xzy.
\end{align*}
Thus we get $xyz \approx xzy$ and so $S \in \lnbd$. Also it follows that $S \in \vdd$.
\end{proof}

Thus the multiplicative reduct $(S, \cdot)$ of every idempotent semiring $S \in \mathbf{L}^{\bullet}$ is a left normal band. Similarly, the multiplicative reduct $(S, \cdot)$ of every idempotent semiring $S \in \mathbf{R}^{\bullet}$ is a right normal band.

In $\mathbf{D}^{\bullet}$ we have the following derivation:
\begin{align*}
& x \approx xyx+x+xyx \\
\Rightarrow \, & x \approx x+xyx+x+xyx+x \\
\Rightarrow \, & x \approx x+xyx+x
\end{align*}

\begin{Theorem}
An idempotent semiring $S\in \mathbf{D}^{\bullet}$ if and only if $S$ is a spined product of an idempotent semiring $S_1\in \mathbf{L}^{\bullet}$ and an idempotent semiring $S_2\in \mathbf{R}^{\bullet}$ with respect to a distributive lattice $D$.
\end{Theorem}
\begin{proof}
Let $S \in \mathbf{D}^{\bullet}$. Then the multiplicative reduct $(S, \cdot)$ is a regular band. So both $\ld$ and $\rd$ are congruences on the multiplicative reduct $(S, \cdot)$. Let $a, b, c \in S$ and $a \ld b$. Then $a=ab, b=ba$ and $ca \ld cb, ac \ld bc$.
Hence we have,
\begin{align*}
(a+c)(b+c) &= ab+ac+cb+c \\
           &= a+ac+cbca+c \hspace{2cm} \textrm{(since $ca \ld cb$)} \\
       &= a+ac+(c+cac+c)bca+c \\
       &= a+ac+cbca+cacbca+cbca+c \\
       &= a+ac+cbca+ca+cbca+c \hspace{2cm} \textrm{(since $ca \ld cb$ and $\ld \subseteq \dd$)}\\
       &= a+ac+(cbc+c+cbc)a+c \\
       &= a+ac+ca+c \\
       &= (a+c)^2 \\
       &= a+c
\end{align*}
and similarly $(b+c)(a+c)=b+c$. Thus $(a+c) \ld (b+c)$. Similarly $(c+a) \ld (c+b)$. Hence $\ld$ is a congruence on $S$. Now for any $a, b \in S$,
\begin{align*}
& a=aba+a+aba \\
\Rightarrow \, & a=a(ba+a+ba)
\end{align*}
implies that $a \ld (ba+a+ba)$. Hence $S/\ld \in \mathbf{R}^{\bullet}$. Similarly $\rd$ is a congruence on $S$ and $S/\rd \in \mathbf{L}^{\bullet}$.
Denote $S_1=S/\ld$ and $S_2=S/\rd$. Also $D=S/\mathcal{D}^{\bullet}$ is a distributive lattice. Since $\ld$, $\rd \subseteq \mathcal{D}^{\bullet}$, it follows that $\phi_1:S_1\longrightarrow D$ and $\phi_2:S_2\longrightarrow D$ defined by $\phi_1(L_a)=D_a$ and $\phi_2(R_a)=D_a$ are well defined surjective homomorphisms, where $L_a \,[ R_a, D_a]$ is the $\ld$ [ $\rd$, $\mathcal{D}^{\bullet} ]$ -class containing $a$. Thus $R = \{(L_a, R_a) \in S_1 \times S_2 \mid a \mathrel\mathcal{D}^{\bullet} b \}$ is a spined product of $S_1$ and $S_2$ with respect to $D$.

Then the mapping   $\theta : S \longrightarrow R$ defined by $\theta(a)=(L_a, R_a)$ is a monomorphism. Again if $( L_a, R_a) \in R$, then we have $a \mathrel\mathcal{D}^{\bullet} b$. Also since $\mathcal{D}^{\bullet}=\ld o \rd$, there exists $c\in S$ such that $a\mathrel\ld c$ and $c \mathrel\rd b$. This implies that $\theta(c)=(L_a, R_a)$. Thus $\theta$ is onto.
\end{proof}

Thus $ \mathbf{D}^{\bullet} \subseteq \mathbf{L}^{\bullet} \vee \mathbf{R}^{\bullet} $. Also $\mathbf{L}^{\bullet} \subseteq \mathbf{D}^{\bullet}$ and $\mathbf{R}^{\bullet} \subseteq \mathbf{D}^{\bullet}$ implies that $ \mathbf{L}^{\bullet} \vee \mathbf{R}^{\bullet} \subseteq \mathbf{D}^{\bullet}$. Then we have the following corollary.

\begin{Corollary}
$ \mathbf{D}^{\bullet} = \mathbf{L}^{\bullet} \vee \mathbf{R}^{\bullet} $.
\end{Corollary}

\begin{center} {\bf PROBLEMS} \end{center}
Characterize the class of all idempotent semirings $S$ such that $\sigma$ is the least distributive lattice congruence on $S$.

\bibliographystyle{amsplain}

\end{document}